\documentclass[12pt]{article}
\usepackage[english]{babel}
\usepackage{amsfonts,amsmath,amsthm,graphicx,a4wide}

\newtheorem{thm}{Theorem}[section]%
\newtheorem{lem}{Lemma}[section]

\newtheorem{ass}{Assumption}

\newcommand{\RR}{{\mathbb R}}
\newcommand{\Id}{{1\!\!1}}

\newcommand{\EE}[1]{{\text{\normalsize$\mathbb E$}}\left(#1\right)}

\newcommand{\CC}{{\mathbb C}}
\newcommand{\DD}{{\mathcal D}}
\newcommand{\HH}{{\mathcal H}}
\newcommand{\MM}{{\mathbb M}}
\newcommand{\UU}{{\mathbb U}}
\newcommand{\II}{{\Id}}

\newcommand{\bra}[1]{\left<#1\right|}
\newcommand{\ket}[1]{\left|#1\right>}
\newcommand{\tr}[1]{\text{Tr}\left(#1\right)}

\newcommand{\half}{\text{\scriptsize $\frac{1}{2}$}}

\newcommand{\nmax}{n^{\text{\tiny max}}}

\title{Design of Strict Control-Lyapunov Functions \\for Quantum Systems with QND Measurements
    \thanks{This work was   supported in part by the "Agence Nationale de la Recherche" (ANR), Projet Jeunes Chercheurs EPOQ2 number ANR-09-JCJC-0070 and Projet Blanc  CQUID number 06-3-13957.}}

\author{ H. Amini\thanks{Mines ParisTech,  Centre Automatique et Syst\`emes, Math\'{e}matiques et Syst\`{e}mes,
        60 Bd Saint Michel, 75272 Paris cedex 06, France,
        {\tt\small hadis.amini@mines-paristech.fr}}
        \and
        P.  Rouchon\thanks{Mines ParisTech,  Centre Automatique et Syst\`emes, Math\'{e}matiques et Syst\`{e}mes,
        60 Bd Saint Michel, 75272 Paris cedex 06, France,
        {\tt\small pierre.rouchon@mines-paristech.fr}}
        \and
        M. Mirrahimi \thanks{INRIA Rocquencourt,
        Domaine de Voluceau, B.P. 105, 78153 Le Chesnay cedex, France,
        {\tt\small mazyar.mirrahimi@inria.fr}}
}

\begin{document}

\maketitle \thispagestyle{empty} \pagestyle{empty}



\begin{abstract}
We consider  discrete-time quantum systems  subject to   Quantum Non-Demolition (QND) measurements and controlled by an adjustable unitary evolution between two successive QND measures. In open-loop, such QND measurements provide  a non-deterministic preparation tool exploiting the back-action of the measurement on the quantum state. We propose here a systematic method based on elementary graph theory and inversion of Laplacian matrices  to construct strict control-Lyapunov functions. This yields  an appropriate feedback law that stabilizes globally  the system towards a chosen target state among the open-loop stable ones, and that makes in closed-loop this preparation deterministic. We  illustrate such  feedback laws through simulations corresponding to an  experimental setup with QND photon counting.
\end{abstract}

\section{Introduction}

Feedback stabilization of quantum states is closely related to the concept of Quantum Non-Demolition (QND) measurement~\cite{braginsky-vorontosov:75,thorne-et-al:78,Unruh-78}. Indeed, as soon as we are interested in applying a measurement-based feedback to stabilize a quantum state, we need to make sure that the measurement itself is not changing the desired target state. This means that the measurement procedure is QND with respect to the projection over the target state. In fact, very often a well-chosen QND measurement protocol can itself be considered as a preparation tool for various quantum states. However, this preparation is generally non-deterministic and one can not make sure to converge towards the desired state except by repeating the experiment many times. The feedback can be applied here to make this process deterministic~\cite{stockton-et-al:PRA2004,mirrahimi-handel:siam07,dotsenko-et-al:PRA09}.

This paper is a generalization  of the feedback law,  proposed  in~\cite{dotsenko-et-al:PRA09,mirrahimi-et-al:cdc09} and experimentally tested in~\cite{Sayrin-et-al:nature2011},   to  generic discrete-time quantum systems  where,  between two successive QND measurements, a controlled unitary evolution can be  applied. The dynamics of such discrete-time  quantum systems are governed  by non linear Markov chains. In~\cite{dotsenko-et-al:PRA09,mirrahimi-et-al:cdc09} the feedback laws were obtained by maximizing the fidelity with respect to  the target state at each time-step: this means that the feedback strategy  was based on the same Lyapunov function given by the fidelity between the current and the target state.   We propose here a systematic and explicit method to design a new  family of  control-Lyapunov functions. The main interest of these new Lyapunov functions relies on the crucial fact that the increases of their expectation values at  step $k+1$ knowing the state  $\rho_k$ at step $k$   remain  strictly positive when $\rho_k$  does not coincide with the  target state.   In closed-loop,  these Lyapunov functions become strict and the   convergence analysis is notably simplified since invariance principles are not necessary. The construction of these strict Lyapunov functions is based on the Hamiltonian $H$ underlying the controlled unitary evolution and relies on  the connectivity of the  graph attached to $H$. They are obtained  by inverting a Laplacian matrix derived from $H$ and the quantum states that are untouched by the QND measurements.

In section~\ref{sec:model}, we describe the finite dimensional  Markovian model together with the main modeling assumptions.
Section~\ref{sec:openloop} is devoted to the open-loop behavior (Theorem~\ref{thm:first}) that can be seen as a  non-deterministic protocol for  preparing a finite number of isolated and orthogonal  quantum states.
In Section~\ref{sec:closedloop}, we  present  the main ideas underlying the construction of these strict control-Lyapunov functions $W_\epsilon$. Then we  define the connectivity graph,  the Laplacian matrix  attached to $H$ and two  technical lemmas  used during this construction. Finally  Theorem~\ref{thm:main} describes
the stabilizing feedback derived from $W_\epsilon$. Closed-loop simulations corresponding to an experimental setup at Ecole Normale Sup\'{e}rieure  are sketched in section~\ref{sec:simul}.

The authors thank M. Brune, I.~Dotsenko, S.~Gleyzes, S.~Haroche and J.M. Raimond for enlightening discussions and references. Advices of L. Praly concerning control-Lyapunov functions  are also acknowledged.

\section{The non-linear Markov model} \label{sec:model}

We consider a finite dimensional quantum system (the underlying Hilbert space $\HH=\CC^d$ is of  dimension $d>0$)  being measured through a generalized measurement procedure placed discretely in time. Between two measurements, the  system undergoes a unitary evolution depending on a scalar control input $u\in\RR$. The dynamics of such discrete time quantum systems  is described by a non-linear controlled  Markov chain whose structure is derived from  quantum physics. We just sketch here this structure with a mathematical viewpoint. A tutorial physical exposure can be found in~\cite{haroche-raimond:book06}.

The system state is described by the density operator $\rho$ belonging to $\DD(\HH)$ the set of positive, Hermitian matrices of trace one:
\begin{equation*}
\DD(\HH):=\{\rho\in\CC^{ d\times d}~|\quad\rho=\rho^\dag,\quad\tr{\rho}=1,\quad\rho\geq 0\}.
\end{equation*}
The generalized measurement procedure admits  $m>0$ different discrete values $\mu\in\{1,\ldots,m\}$: to each measurement outcome $\mu$  is attached a Kraus operator described by a matrix   $M_{\mu}\in \CC^{ d\times d}$. The  Kraus operators $(M_\mu)_{\mu\in\{1,\ldots,m\}}$ satisfy the constraint $\sum_{\mu=1}^{m}M_{\mu}^{\dag}M_{\mu}=\II$ where $\II$ is the identity matrix. In general the $M_\mu$ are not necessarily Hermitian.
The controlled  evolution between two measures $U_u$ is defined by the unitary operator  $\exp(-iu H)=U_u$ where $H$ is a Hermitian operator $H\in\CC^{ d\times d}$ with $H^\dag=H$.

The random evolution of the state $\rho_k\in\DD(\HH)$ at time step $k$ is  modeled through the following Markov process: \begin{equation}\label{eq:dynamic}
 \rho_{k+1}=\UU_{u_k}(\MM_{\mu_k}(\rho_k)),
\end{equation}
where
\begin{itemize}
\item $u_k\in\RR$ is the control at step $k$,
\item $\mu_k$ is a random variable taking values $\mu$ in $\{1,\cdots,m\}$  with probability $p_{\mu,\rho_k}=\tr{M_{\mu}\rho_k M_{\mu}^\dag}$,
\item $\UU_{u}$ is the super-operator
$$ \UU_{u}:\quad \DD(\HH)\ni \rho \mapsto U_{u}\rho U_{u}^\dag\in\DD(\HH),$$

\item For each $\mu$, $\MM_{\mu}$ is the super-operator
$$
\MM_{\mu}:\quad \rho \mapsto\tfrac{M_{\mu}\rho M_{\mu}^\dag}{\tr{M_{\mu}\rho M_{\mu}^{\dag}}}
\in \DD(\HH)$$ defined  for  $\rho\in\DD(\HH)$ such that $\tr{M_{\mu}\rho M_{\mu}^\dag} \neq 0$.
\end{itemize}
We suppose throughout this paper that the two following assumptions are verified
by the system under consideration.

\begin{ass}\label{ass:imp}
The measurement operators $M_\mu$ are diagonal in the same orthonormal  basis $\{\, \ket{n}\, |\quad n\in\{1,\cdots, d\}\}$, therefore $M_\mu=\sum_{n=1}^{ d}c_{\mu,n}\ket{n}\bra{n}$
with $c_{\mu,n}\in\CC.$
\end{ass}
\begin{ass}\label{ass:impt}
 For all $n_1\neq n_2$ in $\{1,\cdots, d\},$ there~exists a $\mu\in\{1,\cdots,m\}$ such that $|c_{\mu,n_1}|^2\neq |c_{\mu,n_2}|^2.$
\end{ass}

 Assumption~\ref{ass:imp} means that the considered measurement process   achieves a   Quantum Non Demolition (QND) measurement  for the physical observables given by orthogonal projections   over the states  $\bigl\{\ket{n}\,|\quad n\in\{1,\cdots, d\}\bigr\}.$  This implies that for  $u_k\equiv 0$, any $\rho=\ket{n}\bra{n}$ corresponding to the orthogonal projector on the basis vector $\ket{n}$, $n\in\{1,\ldots, d\}$, is a fixed point of the Markov process~\eqref{eq:dynamic}. Since the operators $M_\mu$ must satisfy $\sum_{\mu=1}^{m}M_{\mu}^{\dag}M_{\mu}=\II$, we have, according to assumption~\ref{ass:imp}, $\sum_{\mu=1}^{m}|c_{\mu,n}|^2=1$ for all $n\in\{1,\cdots, d\}.$

Assumption~\ref{ass:impt} means that there exists a $\mu$ such that the statistics  when  $u_k\equiv 0$  for obtaining the measurement result $\mu$ are different for the fixed points $\ket{n_1}\bra{n_1}$  and $\ket{n_2}\bra{n_2}$.  This follows by noting that $\tr{M_{\mu}\ket{n}\bra{n}M_{\mu}^\dag}=|c_{\mu,n}|^2$ for $n\in\{1,\cdots, d\}$.

\section{Convergence of the open-loop dynamics}\label{sec:openloop}
When the control vanishes ($u_k=0,~\forall k$), the dynamics is simply given by
\begin{equation}\label{eq:dynopen}
\rho_{k+1} = \MM_{\mu_k}(\rho_k)
\end{equation}
where $\mu_k$ is a random variable with discrete values in $\{1,\ldots,m\}$. The probability $p_{\mu,\rho_k}$ to have $\mu_k=\mu$ depends on $\rho_k$: $p_{\mu,\rho_k}=\tr{M_{\mu}\rho_k M_{\mu}^\dag}$.
  We have then the following theorem characterizing the open-loop asymptotic behavior:
\begin{thm}\label{thm:first}
Consider a Markov process $\rho_k$ obeying the dynamics of~\eqref{eq:dynopen} with an initial condition $\rho_0$ in $\DD(\HH)$. Then
\begin{itemize}
\item with probability one, $\rho_k$ converges  to one of the $ d$ states $\ket{n}\bra{n}$ with $n\in\{1,\cdots, d\}.$
\item  the probability of convergence towards the state $\ket{n}\bra{n}$ depends only on the initial condition $\rho_0$ and is given by
\[\tr{\rho_0\ket{n}\bra{n}}=\bra{n}\rho_0\ket{n}.\]
\end{itemize}
\end{thm}
The proof is a generalization of the one given in~\cite{amini-et-al-10}.
\begin{proof}
For any $n\in\{1,\ldots,d\}$, $\tr{\ket{n}\bra{n}\rho}$ is a martingale. This results from
\begin{multline*}
\EE{\tr{\ket{n}\bra{n}\rho_{k+1}}|\rho_k}\\=\sum_{\mu=1}^{m}\tr{M_{\mu}\rho_k M_{\mu}^\dag}\tr{\ket{n}\bra{n}\MM_{\mu}(\rho_k)}\\
=\sum_{\mu=1}^{m}\tr{\ket{n}\bra{n}M_{\mu}\rho_k M_{\mu}^\dag}=\tr{\sum_{\mu=1}^{m}M_{\mu}^{\dag} M_{\mu}\ket{n}\bra{n}\rho_k}\\
=\tr{\ket{n}\bra{n}\rho_k},
\end{multline*}
where  we have used  the facts that $M_{\mu}$ and $\ket{n}\bra{n}$ commute and that $\sum_{\mu=1}^{m} M_{\mu}^\dag M_{\mu}=\Id$.
 Take the function
\begin{equation}\label{eq:V}
V(\rho):=\sum_{n=1}^{ d}f(\tr{\ket{n}\bra{n}\rho}).
\end{equation}
where $f(x)=\tfrac{x^2}{2}$.
The function $f$ being convex and each $\tr{\ket{n}\bra{n}\rho}$ being a martingale, we infer that $V(\rho)$ is a sub-martingale
\begin{equation*}
\EE{V(\rho_{k+1})|\rho_k}\geq V(\rho_k).
\end{equation*}
More precisely, we have
\begin{multline*}
\EE{V(\rho_{k+1})|\rho_k}=\\
\sum_{n=1}^{ d}\sum_{\mu\in I_{\rho_k}}\tr{M_{\mu}\rho_k M_{\mu}^\dag}f\left(\tfrac{\tr{\ket{n}\bra{n}M_{\mu}\rho_k M_{\mu}^\dag}}{\tr{M_{\mu}\rho_k M_{\mu}^\dag}}\right),
\end{multline*}
with $I_{\rho_k}=\Bigl\{\mu\in\{1,\cdots,m\}~|~\tr{M_{\mu}\rho_k M_{\mu}^\dag}\neq 0\Bigr\}$.
For all sequence of reals $x_1,\cdots,x_m$ and $\theta_1,\cdots,\theta_m$ in the interval $[0,1]$ with $\sum_{\mu=1}^{m}\theta_\mu=1$, we have the identity
\begin{align}\label{eq:main}
\sum_{\mu=1}^{m}\theta_\mu f(x_\mu)=f\left(\sum_{\mu=1}^{m}\theta_\mu x_\mu\right)+\sum_{\mu,\nu=1}^{m} \theta_\mu\theta_\nu\tfrac{(x_\mu-x_\nu)^2}{4}.
\end{align}
For each $\mu$, let $\theta_\mu=\tr{M_{\mu}\rho_k M_{\mu}^\dag}$, $
 x_\mu=\tfrac{\tr{\ket{n}\bra{n}M_{\mu}\rho_k M_{\mu}^\dag}}{\tr{M_{\mu}\rho_k M_{\mu}^\dag}}$ if $\theta_\mu >0$ and
 $x_\mu=0$ otherwise.
Identity~\eqref{eq:main} yields
\begin{multline}\label{eq:res}
\EE{V(\rho_{k+1})|\rho_k}-V(\rho_k)=\\
 \tfrac{1}{4}\sum_{n=1}^{ d}\sum_{\mu,\nu\in I_{\rho_k}}\tr{M_{\mu}\rho_k M_{\mu}^\dag}\tr{M_{\nu}\rho_k M_{\nu}^\dag}\ldots \\ \ldots
\left(\tfrac{|c_{\mu,n}|^2\bra{n}\rho_k\ket{n}}{\tr{M_{\mu}\rho_k M_{\mu}^\dag}}-\tfrac{|c_{\nu,n}|^2\bra{n}\rho_k\ket{n}}{\tr{M_{\nu}\rho_k M_{\nu}^\dag}}\right)^2.
\end{multline}
We have  used the fact that
$\bra{n}\MM_{\mu}(\rho_k)\ket{n}=\frac{|c_{\mu,n}|^2\bra{n}\rho_k\ket{n}}{\tr{M_{\mu}\rho_k M_{\mu}^\dag}}$.
Thus, $V(\rho_k)$ is a sub-martingale, and in addition we have a precise bound on the difference $\EE{V(\rho_{k+1})|\rho_k}-V(\rho_k)$. We denote by $Q(\rho_k)$ the right term in Equation~\eqref{eq:res}. Note that the sum in the definition of $Q(\rho_k)$ is over all $\mu,\nu \in I_{\rho_k}$. However, we can assume that the sum is actually over all $\mu,\nu$ in $\{1,\cdots,m\}$ by observing the following facts.

For any $\mu,\nu$, take  the  mapping
{\small $$
\rho \mapsto \tr{M_{\mu}\rho M_{\mu}^\dag}\tr{M_{\nu}\rho M_{\nu}^\dag}
\left(\tfrac{|c_{\mu,n}|^2\bra{n}\rho \ket{n}}{\tr{M_{\mu}\rho M_{\mu}^\dag}}-\tfrac{|c_{\nu,n}|^2\bra{n}\rho\ket{n}}{\tr{M_{\nu}\rho M_{\nu}^\dag}}\right)^2
$$}
defined only when  $\tr{M_{\mu}\rho M_{\mu}^\dag}\tr{M_{\nu}\rho M_{\nu}^\dag} >0$. Since this mapping is positive and
bounded by $\rho \mapsto\tr{M_{\mu}\rho M_{\mu}^\dag}\tr{M_{\nu}\rho M_{\nu}^\dag}$, it can be extended by continuity to any $\rho\in\DD(\HH)$ by  taking a null value when $\tr{M_{\mu}\rho M_{\mu}^\dag}\tr{M_{\nu}\rho M_{\nu}^\dag}=0$. Thus
\begin{multline}\label{eq:Q}
    Q(\rho)= \tfrac{1}{4}\sum_{n=1}^{ d}\sum_{\mu,\nu=1}^{m}
    \tr{M_{\mu}\rho M_{\mu}^\dag}\tr{M_{\nu}\rho M_{\nu}^\dag}\\
\left(\tfrac{|c_{\mu,n}|^2\bra{n}\rho\ket{n}}{\tr{M_{\mu}\rho M_{\mu}^\dag}}-\tfrac{|c_{\nu,n}|^2\bra{n}\rho\ket{n}}{\tr{M_{\nu}\rho M_{\nu}^\dag}}\right)^2
\end{multline}
 is continuously defined for any $\rho\in\DD(\HH)$  and still satisfies
$$
\EE{V(\rho_{k+1})|\rho_k}-V(\rho_k)= Q(\rho_k), \forall \rho_k\in\DD(\HH).
$$
By Theorem~\ref{thm:app} of the Appendix, the $\omega$-limit set $\Omega$ (in the sense of almost sure convergence), for the trajectories $\rho_k$, is a subset of the set $\{\rho\in\DD(\HH)|\quad Q(\rho)=0\}.$

Let us consider a density matrix $\rho_\infty$ in the $\omega$-limit set. Therefore $Q(\rho_\infty) =0$ implies
\begin{equation}\label{eq:ratio}
\tfrac{|c_{\mu,n}|^2\bra{n}\rho_\infty \ket{n}}{\tr{M_{\mu}\rho_\infty M_{\mu}^\dag}}-\tfrac{|c_{\nu,n}|^2\bra{n}\rho_\infty\ket{n}}{\tr{M_{\nu}\rho_\infty M_{\nu}^\dag}}\, =0,
\end{equation}
for all $\mu$ and $\nu$ in $I_{\rho_\infty}$ and for each basis element $\ket n$.
Since $\tr{\rho_\infty}=1,$ there is at least one  $\bar n$ such that $\bra{\bar n}\rho_\infty \ket{\bar n} >0$. Then
 Equation~\eqref{eq:ratio} simplifies to
\begin{align*}
\tr{M_{\mu}\rho_\infty M_{\mu}^\dag}|c_{\nu,\bar n}|^2=
\tr{M_{\nu}\rho_\infty M_{\nu}^\dag}|c_{\mu,\bar n}|^2,
\end{align*}
for all $\mu,\nu \in I_{\rho_\infty}$. Summing over $\nu \in I_{\rho_\infty}$, we find
\begin{multline*}
\tr{M_{\mu}\rho_\infty M_{\mu}^\dag} \, \Bigl(\sum_{\nu \in I_{\rho_\infty}}|c_{\nu,\bar n}|^2 \Bigr)=\\
\Bigl(\sum_{\nu \in I_{\rho_\infty}}\tr{M_{\nu}\rho_\infty M_{\nu}^\dag} \Bigr)\,|c_{\mu,\bar n}|^2
\end{multline*}
By definition of $I_{\rho_\infty}$
$$
\sum_{\nu \in I_{\rho_\infty}}\tr{M_{\nu}\rho_\infty M_{\nu}^\dag} =\sum_{\nu=1}^m\tr{M_{\nu}\rho_\infty M_{\nu}^\dag}
=1.
$$
Thus we have
$$
\tr{M_{\mu}\rho_\infty M_{\mu}^\dag} \Bigl(\sum_{\nu \in I_{\rho_\infty}}|c_{\nu,\bar n}|^2 \Bigr)=|c_{\mu,\bar n}|^2
.
$$
For $\nu\notin I_{\rho_\infty}$, $c_{\nu,\bar n}=0$ since
$ \sum_{n=1}^{d} |c_{\nu,n}|^2 \bra{n}\rho_\infty\ket{n}=0$,
each $\bra{n}\rho_\infty\ket{n} \geq 0$ and $\bra{\bar n}\rho_\infty\ket{\bar n} >0$. Thus
$\sum_{\nu \in I_{\rho_\infty}}|c_{\nu,\bar n}|^2= \sum_{\nu=1}^m |c_{\nu,\bar n}|^2=1$. Finally, we have
\begin{equation}\label{eq:coef}
\forall \mu\in I_{\rho_\infty}, \quad
\tr{M_{\mu}\rho_\infty M_{\mu}^\dag}=|c_{\mu,\bar n}|^2.
\end{equation}
as soon as $\bra{\bar n} \rho_\infty \ket{\bar n} >0$.

Assume now that exist $\bar n_1\neq \bar n_2$ in $\{1,\ldots,d\}$ such that
$\bra{\bar n_1} \rho_\infty \ket{\bar n_1} >0$ and $\bra{\bar n_2} \rho_\infty \ket{\bar n_2} >0$. Then~\eqref{eq:coef} implies that
\begin{equation} \label{eq:coefbis}
\forall \mu\in I_{\rho_\infty}, \quad
|c_{\mu,\bar n_1}|^2=|c_{\mu,\bar n_2}|^2.
\end{equation}
By Assumption~\ref{ass:impt}, there exists a $\bar \mu\in\{1,\cdots,m\}$ such that
$|c_{\bar\mu,\bar n_1}|^2\neq|c_{\bar\mu,\bar n_2}|^2$.
These terms cannot be simultaneously zero, thus $\tr{M_{\bar \mu}\rho_\infty M_{\bar \mu}^\dag}>0$,   $\bar \mu \in I_{\rho_\infty}$. This is  in  contradiction with~\eqref{eq:coefbis}.
 This closes the proof of the assertion:  the $\omega$-limit set is reduced to the set fixed point  $\ket{n}\bra{n},$ with $n\in\{1,\cdots, d\}$.

We have shown that the probability measure associated to the random variable $\rho_k$ converges to the probability measure
$
\sum_{n=1}^{ d}p_n\delta_{\ket{n}\bra{n}},
$
where $\delta_{\ket{n}\bra{n}}$ denotes the Dirac measure at $\ket{n}\bra{n}$ and $p_n$ is the probability of convergence towards $\ket{n}\bra{n}.$ In particular, we have
$
\EE{\tr{\rho_k\ket{n}\bra{n}}}\longrightarrow p_n.
$
But $\tr{\ket{n}\bra{n}\rho_k}$ is  a martingale thus  $\EE{\tr{\ket{n}\bra{n}\rho_k}}=\EE{\tr{\ket{n}\bra{n}\rho_0}}$ and consequently $p_n=\bra{n}\rho_0\ket{n}$.
\end{proof}

\section{Feedback stabilization}\label{sec:closedloop}

\subsection{Design of strict control-Lyapunov functions}
The goal is to design a feedback law that globally stabilizes the Markov chain~\eqref{eq:dynamic} towards a chosen target state $\ket{\bar n}\bra{\bar n}$, for some $\bar n$  among $\{1,\cdots, d\}$. In previous  publications~\cite{mirrahimi-handel:siam07,mirrahimi-et-al:cdc09,dotsenko-et-al:PRA09}, the proposed feedback schemes   tend to increase at each step the same open-loop martingale $V_{\bar n}(\rho)=\bra{\bar n} \rho \ket{\bar n}$: $u$ is chosen in order to increase $u \mapsto V_{\bar n} (\UU_u(\rho))$.  When $\rho=\ket n\bra n$ with  $n\neq \bar n$, $u \mapsto V_{n}(\rho)$ is minimum at $u=0$. Consequently, its first-order $u$-derivative vanishes at $u=0$. But its   second-order $u$-derivative could also vanishe  at $u=0$. For the photon-box  considered in~\cite{dotsenko-et-al:PRA09} (see also section~\ref{sec:simul}),  this happens  when $|n-\bar n| \geq 2$. Such lack of strong convexity  in $u$ when the image of $\rho$ is almost orthogonal to $\ket{\bar n}$, explains  the fact that, in previous works, the control $u$ is set  to a constant non-zero value when $V_{\bar n}(\rho)$ is close to $0$,

To improve convergence and avoid such constant feedback zone,   we propose  to modify $V_{\bar n}$ using the other open-loop martingales $V_n(\rho)=\bra n \rho \ket n$ and  the sub-martingale
$V(\rho)=\sum_{n=1}^d (\bra n \rho \ket n)^2$ used during the proof of theorem~\ref{thm:first}.  The goal of such modification  is to get  control Lyapunov functions  still admitting a unique global  maximum at $\ket{\bar n}\bra{\bar n}$ but being  strongly convex versus $u$ around $0$  when  $\rho$ is  close to any $\ket n\bra n$,  $n\neq \bar n$.

Take  $W_0(\rho)=\sum_{n=1}^d \sigma_n \bra n \rho \ket n$ with real coefficients $\sigma_n$ to be chosen such that $\sigma_{\bar n}$ remains   the largest one and such that, for any $n \neq \bar n$, the second-order $u$-derivative of $W_0(\UU_u(\ket n \bra n))$ at  $u=0$ is strictly positive. This yields to a set of linear equations (see lemma~\ref{lem:second}) in $\sigma_n$ that can be solved by inverting a Laplacian matrix (see lemma~\ref{lem:first}). Notice that $W_0$ is an open-loop martingale. To obtain a sub-martingale we consider (see theorem~\ref{thm:main}) $W_\epsilon(\rho)= W_0(\rho) + \epsilon V(\rho)$. For $\epsilon>0$ small enough:  $W_\epsilon(\rho)$ still admits a unique global maximum at $ \ket{\bar n}\bra{\bar n}$; for $u$ close to $0$,
$W_\epsilon(\UU_u(\ket n \bra n))$ is strongly convex for any $n\neq \bar n$ and strongly concave for $n=\bar n$. This implies that  $W_\epsilon$ is a control-Lyapunov function with arbitrary small control (see proof of theorem~\ref{thm:main}).

Let us continue  by some definitions and lemmas that underlay the construction of these  strict control-Lyapunov functions $W_\epsilon$.

\subsection{Connectivity graph and Laplacian matrix}
To the Hamiltonian operator $H$ defining the controlled unitary evolution $U_u=e^{-\imath u H}$, we associate its  undirected connectivity graph denote by $G^H$. This graph  admits  $d$ vertices labeled by $n\in\{1,\ldots,d\}$. Two different vertices  $n_1\neq n_2$ ($n_1,n_2\in\{1,\ldots,d\}$) are linked by an  edge, if and only if, $\bra{n_1}H\ket{n_2} \neq 0$. Attached to $H$, we also associate $R^H$, the real symmetric matrix $d\times d$ (Laplacian matrix) with entries
{\small\begin{align}\label{eq:matrix}
R^H_{n_1,n_2}=2\Big(\delta_{n_1,n_2}\bra{n_1}H^2\ket{n_2}-|\bra{n_1}H\ket{n_2}|^2 \Big).
\end{align}}
\begin{lem}\label{lem:first}
Assume the  graph $G^H$  to be connected. Then  for any positive reals
$\lambda_n$, $n\in\{1,\ldots, d\}$, $n\neq \bar n$, there exists a vector $\sigma=(\sigma_n)_{n\in\{1,\ldots,d\}}$ of $\RR^d$ such that
$R^H\sigma=-\lambda$
where $\lambda$ is the vector of $\RR^d$ of components $\lambda_n$ for $n\neq \bar n$ and $\lambda_{\bar n} = - \sum_{n\neq \bar n} \lambda_n$.
\end{lem}
\medskip
\begin{proof} Note that $R^H$ is symmetric and the sum of the entries for any column and any row of $R^H$ is equal to zero. Therefore, the vector $(1\cdots 1)^T$ is in the kernel of $R^H$. The diagonal (resp.non-diagonal) components of $R^H$ are positive (resp. negative). Therefore $R^H$ is a Laplacian matrix (see~\cite[Ch. 4]{beineke-wilson:book04}). The connectivity graph associated to $R^H$ coincides with $G^H$. Since this graph is supposed connected, classical results of graph theory (see, e.g.,~\cite[Theorem $3.1$]{beineke-wilson:book04}) imply that  the "constant" vector $(1,\cdots, 1)^T$ spans the kernel of $R^H$. Therefore, the dimension of the image of $R^H$ is equal to $d-1$. Since $R^H$ is symmetric, its image  coincides with the orthogonal to its kernel.
For the sake of completeness, here we give a simple proof of this statement. Indeed, for any vector $X$ in the kernel of $R^H$, we have
$$
\sum_{n_1,n_2\in\{1,\cdots, d\}}R^H_{n_1,n_2}(X_{n_1}-X_{n_2})^2=X^TRX=0.
$$
This implies
$R^H_{n_1,n_2}(X_{n_1}-X_{n_2})^2=0$, $\forall n_1,n_2\in\{1,\cdots, d\}$.
As the  graph of $R^H$ is connected, we necessarily have $X_{n_1}=X_{n_2}$ for all $n_1,n_2\in\{1,\cdots, d\}.$
Thus any vector  orthogonal to  $(1\cdots 1)^T$ is in the image of $R^H$. The vector $\lambda$ is orthogonal to $(1,\cdots,1)^T$.
\end{proof}

\begin{lem}\label{lem:second} Take $\bar n\in\{1,\ldots, d\}$ and consider  $\lambda_n >0$, for $n\in\{1,\ldots,d\}/\{\bar n\}$. Assume $G^H$ connected and consider the vector $\sigma=(\sigma_n) \in\RR^d$ given by Lemma~\ref{lem:first}. For any $\rho\in\DD(\HH)$ we set
\begin{equation}\label{eq:W0}
W_{0}(\rho) =\sum_{n=1}^{ d}\sigma_n\tr{\ket{n}\bra{n}\rho}= \sum_{n=1}^{ d} \sigma_n \bra n \rho \ket n.
\end{equation}
 Then for any $n\in\{1,\ldots,d\}/\{\bar n\}$ we have
 $$
 \left.\tfrac{d^2 W_0\big(\UU_u(\ket{n}\bra{n})\big)}{d u^2}\right|_{u=0}= \lambda_n >0
 $$
 and
 $$
  \left.\tfrac{d^2 W_0\big(\UU_u(\ket{\bar n}\bra{\bar n })\big)}{d u^2}\right|_{u=0}= \lambda_{\bar n }=-\sum_{n\neq\bar n} \lambda_n <0
 $$
\end{lem}
\begin{proof} For any $n$, set $g_n(u)= W_0(\UU_u(\ket{n}\bra{n})=W_0 (e^{-\imath u H}\ket{n}\bra{n} e^{\imath u H}) $.
The Baker-Campbell-Hausdorff formula yields up to third order terms in $u$ ($[\cdot,\cdot]$ is the commutator):
$$
\UU_u(\ket n \bra n)\approx
\ket n \bra n - \imath u [H,\ket n \bra n] - \tfrac{u^2}{2} [H,[H,\ket n \bra n]]
.
$$
Consequently for any $l\in\{1, \ldots, d\}$, we have
{\small
\begin{multline}\label{eq:d2tr}
\tr{\UU_u(\ket{n}\bra{n})\ket{l}\bra{l}}
\approx
\tr{\ket{l}\bra{l}\big(\ket{n}\bra{n}-\imath u[H,\ket{n}\bra{n}]\big)}\\
-\tfrac{u^2}{2}\tr{\ket{l}\bra{l}\big(\left[H,[H,\ket{n}\bra{n}]\right]\big)}\\
=\left(\delta_{l,n}+\tfrac{u^2}{2}\tr{[H,\ket{n}\bra{n}][H,\ket{l}\bra{l}]}\right),
\end{multline}}
since $\tr{\ket{l}\bra{l}[H,\ket{n}\bra{n}]}=-\tr{[\ket{l}\bra{l},\ket{n}\bra{n}] H}=0$ because $\ket{l}\bra{l}$ commutes with $\ket{n}\bra{n}$), and since
{ \small $$
 \tr{\ket{l}\bra{l}[H,[H,\ket{n}\bra{n}]]}
 =-\tr{[H,\ket{n}\bra{n}][H,\ket{l}\bra{l}]}.
$$}
Thus up to third order terms in $u$, we have
{\small
\begin{multline*}
g_n(u)
=\sum_{l=1}^{ d}\sigma_l\left(\delta_{l,n}+\tfrac{u^2}{2}\tr{[H,\ket{n}\bra{n}][H,\ket{l}\bra{l}]}\right),
\end{multline*}}
Therefore:
\begin{multline*}
\left.\tfrac{\partial^2 W_0\big(\UU_u(\ket{n}\bra{n})\big)}{\partial u^2}\right|_{u=0}
=\sum_{l=1}^{ d} \sigma_l\tr{[H,\ket{n}\bra{n}][H,\ket{l}\bra{l}]}.
\end{multline*}
It is not difficult to see that
$\tr{[H,\ket{n}\bra{n}][H,\ket{l}\bra{l}]}=-R^H_{n,l}$
Thus $\left.\tfrac{\partial^2 W_0\big(\UU_u(\ket{n}\bra{n})\big)}{\partial u^2}\right|_{u=0}
=-\sum_{l=1}^{ d} R^H_{n,l} \sigma_l$.
\end{proof}

\subsection{The global stabilizing feedback}

The main result is expressed through the following theorem.
\begin{thm}\label{thm:main}
Consider the controlled Markov chain of state  $\rho_k$ obeying~\eqref{eq:dynamic}.
 Assume that the  graph $G^H$ associated to the Hamiltonian $H$  is connected and that the Kraus operators satisfy assumptions~\ref{ass:imp} and~\ref{ass:impt}. Take $\bar n\in\{1,\ldots, d\}$ and $d-1$ strictly positive real numbers $\lambda_n >0$, $n\in\{1,\ldots,d\}/\{\bar n\}$. Consider the component $(\sigma_n)$ of $\sigma\in\RR^d$ defined by Lemma~\ref{lem:first}.
Denote by $\rho_{k+\half}=\MM_{\mu_k}(\rho_k)$ the quantum state just after the measurement outcome $\mu_k$ at step $k$.
Take $\bar u >0$ and consider  the following feedback law
\begin{equation}\label{eq:control}
u_k=K(\rho_{k+\half})=\underset{u\in[-\bar u,\bar u]}{\text{argmax}}\Big(W_\epsilon\big(\UU_{u}\big(\rho_{k+\half}\big)\big)\Big),
\end{equation}
where the control-Lyapunov function $W_{\epsilon}(\rho)$ is defined by
\begin{equation}\label{eq:Weps}
W_{\epsilon}(\rho)=\sum_{n=1}^d \left(\sigma_n \bra{n}\rho\ket{n} + \tfrac{\epsilon}{4} \left( \bra{n}\rho\ket{n} \right)^2 \right)
\end{equation}
with the   parameter $\epsilon>0$ not too large to ensure that\small
$$
\forall n \in\{1,\ldots,d\}/\{\bar n\}, ~
\lambda_n + \epsilon \left( (\bra n H \ket n )^2 - \bra n H^2 \ket n  \right) >0
.
$$\normalsize
Then, for any $\rho_0\in\DD(\HH)$,  the closed-loop trajectory $\rho_k$ converges almost surely to the pure state $\ket{\bar n}\bra{\bar n}$.
\end{thm}

The proof relies on the fact that $W_{\epsilon}(\rho)$ is  a strict Lyapunov function for the closed-loop system.

\begin{proof} We have
\small
\begin{align*}
&\EE{W_{\epsilon}(\rho_{k+1})|\rho_k}-W_{\epsilon}(\rho_k)=\\
&\qquad \sum_{\mu\in I_{\rho_k}}p_{\mu,\rho_k}\Big(W_{\epsilon}\big(\UU_{K(\MM_\mu(\rho_k))}(\MM_{\mu}(\rho_k))\big)-W_{\epsilon}(\rho_k)\Big)=\\
&\qquad\sum_{\mu\in I_{\rho_k}}p_{\mu,\rho_k}\Big(\underset{u\in[-\bar u,\bar u]}{\text{max}}\Big(W_{\epsilon}\big(\UU_{u}(\MM_{\mu}(\rho_k))\big)\Big)-W_{\epsilon}(\rho_k)\Big)=\\
&\qquad\sum_{\mu\in I_{\rho_k}}p_{\mu,\rho_k}\Big(W_{\epsilon}\big(\MM_{\mu}(\rho_k)\big)-W_{\epsilon}(\rho_k)\Big)+\\
&\qquad\sum_{\mu\in I_{\rho_k}}p_{\mu,\rho_k}\Big(\underset{u\in[-\bar u,\bar u]}{\text{max}}\Big(W_{\epsilon}\big(\UU_{u}(\MM_{\mu}(\rho_k))\big)\Big)-W_{\epsilon}\big(\MM_{\mu}(\rho_k)\big)\Big).
\end{align*}\normalsize
We define the following functions of $\rho_k$,
\small
\begin{equation*}
Q_1(\rho_k):=\sum_{\mu\in I_{\rho_k}}p_{\mu,\rho_k}\Big(W_{\epsilon}\big(\MM_{\mu}(\rho_k)\big)-W_{\epsilon}(\rho_k)\Big),
\end{equation*}
\normalsize and
\small\begin{multline*}
Q_2(\rho_k):=\\
\sum_{\mu\in I_{\rho_k}}p_{\mu,\rho_k}\Big(\underset{u\in[-\bar u,\bar u]}{\text{max}}\Big(W_{\epsilon}\big(\UU_{u}(\MM_{\mu}(\rho_k))\big)\Big)-W_{\epsilon}\big(\MM_{\mu}(\rho_k)\big)\Big).
\end{multline*}
\normalsize
These functions are both positive continuous functions of $\rho_k$ (the continuity of these functions can be proved in the same way as the proof of the continuity of $Q(\rho_k)$ in Theorem~\ref{thm:first}). By Theorem~\ref{thm:app} of the appendix, the $\omega$-limit set $\Omega$ is included in the following set
\begin{equation*}
\{\rho\in\DD(\HH)|~ Q_1(\rho)=0\}~\cap~\{\rho\in\DD(\HH)|~ Q_2(\rho)=0\}.
\end{equation*}
Indeed $Q_1$ coincides with $Q$ defined in~\eqref{eq:Q}. During the proof of Theorem~\ref{thm:first}, we have shown  that $Q_1(\rho)=0$ implies $\rho=\ket{n}\bra{n}$ for some $n\in\{1,\cdots, d\}$ .
But $Q_2(\ket{n}\bra{n})=0$  implies that
$\underset{u\in[-\bar u,\bar u]}{\text{max}}W_{\epsilon}\big(\UU_{u}(\ket{n}\bra{n})\big) = W_{\epsilon}\big(\bra n \ket n\big)$ since $\MM_\mu(\ket n \bra n) =\ket n \bra n$. According to the Lemma~\ref{lem:second} and the relation~\eqref{eq:d2tr}, we have
$ \left.\tfrac{d W_{\epsilon}\big(\UU_{u}(\ket{n}\bra{n})\big)}{du}\right|_{u=0} =0$ and
$$
\left.\tfrac{d^2 W_{\epsilon}\big(\UU_{u}(\ket{n}\bra{n})\big)}{du^2}\right|_{u=0}
= \lambda_n + \epsilon \left( (\bra n H \ket n )^2 - \bra n H^2 \ket n  \right)
.
$$
Since $\lambda_n >0$ for $n\neq \bar n$ and $\epsilon$ is not too large, for any $n\neq \bar n$, $u=0$ is a locally strict minimum of $W_{\epsilon}\big(\UU_{u}(\ket{n}\bra{n})\big)$. Consequently $Q_2(\ket n \bra n)=0$ implies that  $n=\bar n$.
\end{proof}

\section{Closed-loop simulations for the Photon-Box}\label{sec:simul}

 We recall the photon-box model presented in~\cite{dotsenko-et-al:PRA09}: $d=\nmax+1$, where $\nmax$ is the maximum photon number. To be compatible  with usual  quantum-optics notations the ortho-normal basis $\ket{n}$ of the Hilbert space $\HH=\CC^{\nmax+1}$ is indexed by $n\in\{0,\ldots,\nmax\}$ (photon number).   For such system $\mu_k$ takes just two values $g$ or $e,$ and the measurement operators $M_g$ and $M_e$ are defined by $M_g=\cos(\phi_0+\theta \bf N)$ and $M_e=\sin(\phi_0+\theta \bf N)$ ($\phi_0$ and $\theta$ are constant angles fixed as the experiment parameters). When $\theta/\pi$ is irrational,  assumption~\ref{ass:impt} is satisfied.
The photon number operator $\bf N$ is defined by $\bf N=a^\dag a,$ where $\bf a$ is the annihilation operator truncated to $\nmax$ photons. $\bf a$  corresponds to the upper $1$-diagonal matrix filled with $(\sqrt{1},\cdots,\sqrt{\nmax}):$
and $\bf N$ to  the diagonal operator filled with $(0,1, \ldots, \nmax)$.
The truncated creation operator denoted by $\bf a^\dag$ is the Hermitian conjugate of $\bf a.$ Consequently, the Kraus operators $M_g$ and $M_e$ are also diagonal matrices with cosines and sines on the diagonal.

The Hamiltonian $H= \imath (\bf a ^\dag - a)$ yields the unitary operator $U_u=e^{u (\bf a^\dag - a)}$ (also known as displacement operator). Its graph $G^H$ is connected. The associated Laplacian matrix $R^H$  admits a  simple tri-diagonal structure with diagonal elements $R^H_{n,n}= 4 n + 2$,  upper diagonal elements
$R^H_{n-1,n}= -2n$ and under diagonal elements  $R^H_{n+1,n}=-2n-2$ (up to some truncation distortion  for $n=\nmax-1,\nmax$).

For a goal photon number $\bar n\in\{0,\ldots,\nmax-1\}$, we propose the following setting for the $\lambda_n$  and $\epsilon$ defining $W_\epsilon$ of Theorem~\ref{thm:main}:
$$
\forall n\in\{0,\ldots,\nmax\}/\{\bar n\},~\lambda_n=1.
$$
Figure~\ref{fig:Sigma} corresponds to the values of $\sigma_n$ found with  $\lambda_n$ given in above with $\bar n=3$ and $\nmax=10$. We remark that $\sigma_{\bar n}$ is maximal.

Since  $(\bra n H \ket n )^2 - \bra n H^2 \ket n = -2n-1$, the constraint on $\epsilon >0$ imposed by Theorem~\ref{thm:main} reads $\forall n\neq \bar n$, $\epsilon < \frac{1}{2n+1}$, i.e., $\epsilon < \frac{1}{2 \nmax+1}$.

The maximization defining the feedback in Theorem~\ref{thm:main} could be problematic in practice. The unitary propagator $U_u$ does not admit in general a simple analytic form, as it is the case here.  Simulations below  show that we can replace $U_u$ by ist quadratic approximation valid for $u$ small and derived from Baker-Campbell-Hausdorff:
$$
\UU_u(\rho)
= \rho - \imath u [H,\rho] - \tfrac{u^2}{2} [H,[H,\rho]] + O(|u|^3)
.
$$
This yields to an  explicit quadratic approximation of $W_\epsilon(\UU_u(\rho))$ around $0$. The feedback is then given by replacing  $W_\epsilon(\UU_u(\rho))$  by a parabolic expression in $u$  for  the maximization providing the feedback $u_k=K(\rho_{k+\half})$.

For the simulations  below, we take $\nmax=10$, $\bar n=3$ and $\epsilon = \frac{1}{4\nmax + 2}$ and $\bar u=\frac{1}{10}$. The parameters appearing in the Kraus operators are $\theta=\tfrac{\sqrt{2}}{5}$  and $\phi_0=\tfrac{\pi}{4}-\bar n\theta$.
 Figure~\ref{fig:fid} corresponds to $100$ realizations of the closed loop Markov process  with an approximated feedback obtained by the quadratic approximation of $U_u$ versus $u$ sketched here above. Each realization starts with the same initial state $\mathbb U_{\sqrt{\bar n}}(\ket 0 \bra 0)$ (coherent state with an average of  $3$  photon). Despite the quadratic approximation used to derived the feedback, we observe a rapid convergence towards the goal state $\ket{\bar n}\bra{\bar n}$.
 In more realistic simulations not presented  here but  including  the quantum filter to estimate $\rho_k$ from the measurements $\mu_k$ and also the main experimental imperfections described~\cite{dotsenko-et-al:PRA09}, the asymptotic value of the average fidelity is larger than 0.6. This indicates that such feedback laws are robust.

%
%
%
%
%
%
%

 \begin{figure}
  \centerline{\includegraphics[width=0.5\textwidth]{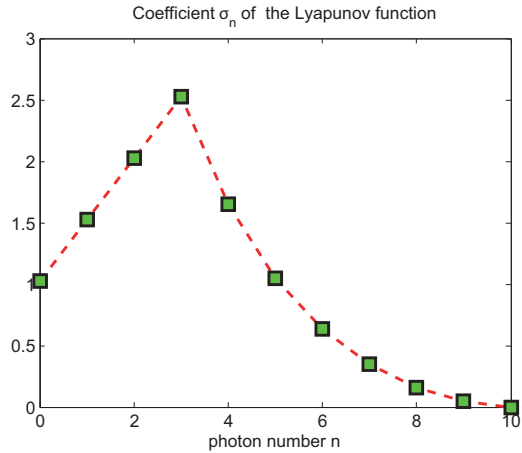}}
   \caption{ The coefficients $\sigma_n$ used for the Lyapunov function
   $W_\epsilon(\rho)$ defined in~Theorem~\ref{thm:main} during  the closed-loop simulations, the goal photon number
   $\bar n=3$ corresponds to the maximum of the $\sigma_n$. }\label{fig:Sigma}
\end{figure}
\begin{figure}
  \centerline{\includegraphics[width=0.5\textwidth]{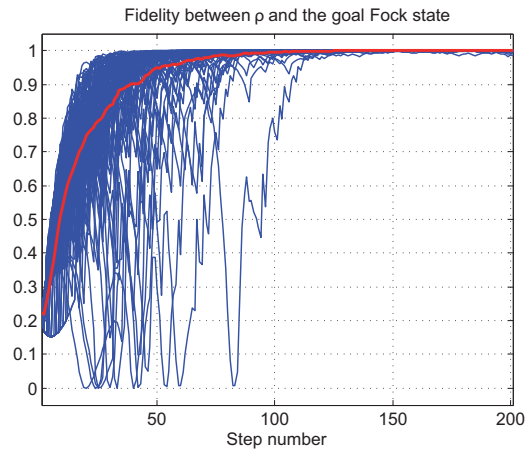}}
   \caption{ $\bra{3}\rho_k\ket{3}$ (Ideal case; fidelity with respect to the $3$-photon goal state) versus the time step $k\in\{0,\cdots,200\}$ for $100$ realizations of the closed-loop Markov process (blue curves) starting from the same coherent state $\rho_0=\UU_{\sqrt{3}}(\ket{0}\bra{0}).$ The ensemble average over these realizations corresponds to the thick red curve.}\label{fig:fid}
\end{figure}

\section{Concluding remarks}

The method proposed here to derive strict control-Lyapunov could certainly be extended to
\begin{itemize}

\item prove exponential closed-loop convergence for the feedback law given by theorem~\ref{thm:main}.

\item  the general situation where the control $u$ appears directly in the  Kraus operators $M_\mu(u)$ instead of being separated from the QND measures and attached to a unitary evolution applied after each measurement.

\item continuous-time quantum systems subject to QND measurements such as those considered in~\cite{mirrahimi-handel:siam07} and described in detail in~\cite{wiseman-milburn:book}.

\item infinite-dimensional quantum systems as in~\cite{somaraju-et-al:cdc2011} that consider  the   photon-box system  without truncation to a finite number of photons.

\end{itemize}


\appendix
\section{Appendix}
\begin{thm}\label{thm:app}
\rm Let $X_k$ be a Markov chain on the compact state space $S.$ Suppose, there exists a non-negative function $V(X)$ satisfying
\begin{equation}\label{eq:most}
\EE{V(X_{k+1})|X_k}-V(X_k)= Q(X_k),
\end{equation}
where  $Q(X)$ is a positif continuous function of $X,$ then the $\omega$-limit set $\Omega$ (in the sense of almost sure convergence) of $X_k$ is included in the following set
\begin{equation*}
I:=\{X|\quad Q(X)=0\}.
\end{equation*}
\end{thm}
\medskip
\begin{proof} The proof is just an application of the Theorem $1$ in~\cite[Ch. 8]{kushner-71}, which shows that $Q(X_k)$ converges to zero for almost all paths. It is clear that the continuity of $Q(X)$ with respect to $X$ and the compactness of $S$ implies that the $\omega$-limit set of $X_k$ is necessarily included into the set I.
\end{proof}
\
\end{document}